%% file: main.tex
\documentclass[a4paper,11pt]{article}

\input{headers}

\begin{document}

\title{Decidability of Being a Union-splitting}
\author{Tenyo Takahashi\footnote{\href{mailto:t.takahashi@uva.nl}{t.takahashi@uva.nl}}}
\affil{\small Institute for Logic, Language and Computation, \\ University of Amsterdam}
\date{}
\maketitle

\begin{abstract}
    Many logical properties are known to be undecidable for normal modal logics, with few exceptions such as consistency and coincidence with $\K$. This paper shows that the property of being a union-splitting in $\NExt{\K}$, the lattice of normal modal logics, is decidable, thus answering the open problem \cite[Problem 2]{handbookModalDecision2007}. This is done by providing a semantic characterization of union-splittings in terms of finite modal algebras. Moreover, by clarifying the connection to union-splittings, we show that in $\NExt{\K}$, having a decidable axiomatization problem and being a (un)decidable formula are also decidable. The latter answers \cite[Problem 17.3]{czModalLogic1997} for $\NExt{\K}$.\footnote{This paper is based on \cite[Chapter 5]{TakahashiThesis}.}
\end{abstract}

\section{Introduction}

Many logical properties have been shown to be undecidable for normal modal logics. The pioneering work by Thomason \cite{Thomason1982} showed the undecidability of Kripke completeness. This was followed by a series of works by Chagrov \cite{chagrov1990I, chagrov1990II, Chagrov2002}, which introduced a general method for showing undecidability. This method can be applied to show the undecidability of various logical properties, including the finite model property, first-order definability, decidability, tabularity, and the coincidence of a fixed tabular logic. For a comprehensive overview and additional references, we refer to \cite{handbookModalDecision2007} and \cite[Chapter 17]{czModalLogic1997}.

Given the generality of Chagrov's method, it might seem that all meaningful logical properties would be undecidable. Indeed, \cite{handbookModalDecision2007} pointed out that ``we know only two interesting decidable properties of finitely axiomatizable logics in $\NExt{\K}$: consistency and coincidence with $\K$.'' 

However, in this paper, we show that the property of being a union-splitting in $\NExt{\K}$ (the lattice of normal modal logics) is decidable, answering the open question \cite[Problem 2]{handbookModalDecision2007} in the affirmative. This property contracts consistency and the coincidence with $\K$ as there are continuum many union-splittings and continuum many non-union-splittings in $\NExt{\K}$. Our proof idea is to give a semantic characterization for a logic $\K + \phi$ to be a union-splitting in terms of finite modal algebras. This also yields the decidability of being a splitting in $\NExt{\K}$. Moreover, we observe that, for a formula $\phi$, $\K + \phi$ has a decidable axiomatization problem iff $\phi$ is a decidable formula iff $\K + \phi$ is a union-splitting or the inconsistent logic. Thus, the decidability of being a union-splitting implies that having a decidable axiomatization problem and being a (un)decidable formula are also decidable. The latter answers \cite[Problem 17.3]{czModalLogic1997} for $\NExt{\K}$ in the affirmative.

This paper is organized as follows. Section 2 introduces preliminaries, in particular, those on the decision problem of logical properties. Section 3 proves the main theorem, the decidability of being a union-splitting in $\NExt{\K}$. Section 4 discusses applications of the decidability result to axiomatization problems and (un)decidable formulas.

\section{Preliminaries}

We assume familiarity with the basics of normal modal logics and algebraic semantics (see, e.g., \cite{czModalLogic1997}). We will only deal with normal modal logics, so we also call them logics. For a logic $L_0$, $\NExt{L_0}$ denotes the complete lattice of all logics containing $L_0$, ordered by inclusion. Recall that $\K$ is the least normal modal logic, and $\NExt{\K}$ is the lattice of all normal modal logics. We will identify a logical property $P$ with the set of logics having $P$, and write $L \in P$ if the logic $L$ has $P$.

We will use the following notations. For a logic $L$, let $\V(L) = \{\A: \A \models L\}$, and for a modal algebra $\A$, let $\Log \A = \{\phi: \A \models \phi\}$. The letters $\H$ and $\S$ respectively denote the closure operator of taking homomorphic images and subalgebras. For a class $\class{K}$ of algebras, let $\class{K}\si$ be the class of subdirectly irreducible (s.i., for short) members of $\class{K}$ and $\class{K}\fsi$ be the finite s.i.~members of $\class{K}$. We abbreviate $\Box \cdots \Box \phi$ ($n$ times $\Box$) as $\Box^n \phi$ and $\phi \land \Box\phi \land \cdots \land \Box^n \phi$ as $\Box^{\leq n} \phi$. Similar abbreviations apply to the modal operation on modal algebras.

\subsubsection*{Decision problem of logical properties}

We refer to \cite[Chapter 17]{czModalLogic1997} and \cite{handbookModalDecision2007} for a detailed introduction and survey of the decision problem of logical properties for modal logic.
Intuitively, the question is: 
\begin{quote}
    Is there an algorithm that, given a modal logic, decides whether it has a specific property?
\end{quote}
There are different ways to formulate this question, depending on how logics are encoded as input. Note that an input must be a finite object, so it is certainly not possible to take all of the continuum many logics into account. The most general possible formulation is to consider all recursively axiomatizable logics, encoded by recursive functions that enumerate their theorems. However, Kuznetsov showed that this only leads to triviality, similar to Rice's Theorem for partial recursive functions. Kuznetsov left the result unpublished, but one can find a proof in \cite[Section 17.1]{czModalLogic1997}. The result also holds for $\NExt{\Kf}$, $\NExt{\Sf}$, and other lattices of normal modal logics.

\begin{theorem}[Kuznetsov]
    Let $P$ be a non-trivial property of recursively axiomatizable logics, that is, there are a recursively axiomatizable logic that has $P$ and a recursively axiomatizable logic that does not have $P$. Then it is undecidable whether a recursively axiomatizable logic has $P$. 
\end{theorem}

So, it is a convention to restrict ourselves to finitely axiomatizable logics. Since most logics we encounter in practice are finitely axiomatizable, this is not a serious drawback. A finitely axiomatizable logic will be encoded by a finite set of formulas axiomatizing the logic, or equivalently, a single formula axiomatizing the logic.

\begin{definition}
    Let $L_0$ be a normal modal logic. A logical property $P$ is \emph{decidable} in $\NExt{L_0}$ iff the set $\{\phi: L_0 + \phi \in P\}$ is decidable.
\end{definition}

Quite a lot of properties are known to be undecidable in $\NExt{\K}$, including Kripke completeness, the finite model property, first-order definability, decidability, and tabularity (see, e.g., \cite{handbookModalDecision2007}). On the contrary, only very few interesting properties were known to be decidable in $\NExt{\K}$. The most interesting ones would be the consistency and the coincidence with $\K$ \cite{handbookModalDecision2007}.

\subsubsection*{Union-splittings and Jankov formulas}

The notion of splittings comes from lattice theory.

\begin{definition}
    Let $X$ be a complete lattice. A \emph{splitting pair} of $X$ is a pair $(x, y)$ of elements of $X$ such that $x \not\leq y$ and for any $z \in X$, either $x \leq z$ or $z \leq y$. If $(x, y)$ is a splitting pair of $X$, we say that $x$ \emph{splits} $X$ and $y$ is a \emph{splitting} in $X$. 
\end{definition}

This notion has been an important tool in the study of lattices of logics. See \cite[Section 10.7]{czModalLogic1997} for a historical overview. 

\begin{definition} 
    Let $L_0$ and $L$ be logics.
    \begin{enumerate}
        \item $L$ is a \emph{splitting} in $\NExt{L_0}$ iff it is a lattice-theoretic splitting in the lattice $\NExt{L_0}$.
        \item $L$ is a \emph{union-splitting} in $\NExt{L_0}$ iff it is the join of a set of splittings in $\NExt{L_0}$.
    \end{enumerate}
\end{definition}

Blok \cite{blok1978degree} identified splittings and union-splittings in $\NExt{\K}$ using \emph{Jankov formulas}, and proved the finite model property of them. An alternative proof can be found in \cite{stablecanonicalrules}. Recall that a modal algebra $\A$ is of \emph{height $\leq n$} if $\A \models \Box^{n+1}\bot$, or equivalently, $\Box^{n+1} 0 = 1$; $\A$ is of \emph{finite height} if it is of height $\leq n$ for some $n \in \omega$.

\begin{definition}
    Let $\A$ be a finite s.i.~modal algebra of height $\leq n$. The \emph{Jankov formula} $\epsilon(\A)$ associated with $\A$ is:
    \begin{align*}
        \epsilon(\A, D) &= (\Box^{n+1} \bot \land \Land \{\Boxx{\leq n} \gamma : \gamma \in \Gamma\}) \to \Lor \{ \Boxx{\leq n} \delta : \delta \in \Delta\} \\
                        &= (\Box^{n+1} \bot \land \Boxx{\leq n} \Land \Gamma) \to \Lor \{ \Boxx{\leq n} \delta : \delta \in \Delta\},
    \end{align*}
    where 
    \begin{align*}
        \Gamma  = & \{p_{a \lor b} \leftrightarrow p_a \lor p_b : a,b \in A\} \cup \\
        & \{\lnot p_a \leftrightarrow \lnot p_a : a \in A\} \cup \\
        & \{\Dia p_a \leftrightarrow p_{\Dia a} : a \in A\} \cup,
    \end{align*}
    and
    \begin{align*}
        \Delta = \{p_a : a \in A, a \ne 1\}.
    \end{align*}
\end{definition}

These formulas are called Jankov formulas because they have a similar semantic characterization as the formulas introduced by Jankov \cite{jankov1963} (see also \cite{dejongh1968}) for superintuitionistic logics.

\begin{proposition}
    Let $\A$ be a finite s.i.~modal algebra of finite height and $\B$ be a modal algebra. Then $\B \not\models \epsilon(\A)$ iff $\A$ is a subalgebra of a s.i. homomorphic image of $\B$.
\end{proposition}

\begin{theorem}[\cite{blok1978degree}] \label{Thm union splitting and Jankov formula}
    Let $L$ be a logic. 
    \begin{enumerate}
        \item $L$ is a splitting in $\NExt{\K}$ iff $L$ is axiomatizable by a Jankov formula of a finite s.i.~modal algebra of finite height. 
        \item $L$ is a union-splitting in $\NExt{\K}$ iff $L$ is axiomatizable by Jankov formulas of finite s.i.~modal algebras of finite height. 
    \end{enumerate}
\end{theorem}

\begin{theorem}[\cite{blok1978degree}] \label{Thn K-union-splittings have fmp}
    Every union-splitting in $\NExt{\K}$ has the finite model property.
\end{theorem}

A finitely axiomatizable logic with the finite model property is decidable, known as Harrop's theorem (see, e.g., \cite[Theorem 16.13]{czModalLogic1997}).

\begin{corollary} \label{Cor fa K-union-splittings are decidable}
    Every finitely axiomatizable union-splitting in $\NExt{\K}$ is decidable.
\end{corollary}

\section[\texorpdfstring{Decidability of being a union-splitting in $\NExt{\K}$}{Decidability of being a union-splitting in NExtK}]{Decidability of being a union-splitting in $\NExt{\K}$} \label{Sec decidability of K-union-splittings}

In this section, we prove the decidability of being a union-splitting in $\NExt{\K}$. We first give a semantic characterization of union-splittings in $\NExt{\K}$. Let $\FH$ be the class of modal algebras of finite height, that is, 
\begin{align*}
    \FH &= \{\A: \exists n \in \omega  (\A \models \Box^n\bot)\} \\
        &= \{\A: \exists n \in \omega  (\Box^n 0 = 1)\}.
\end{align*}

\begin{lemma} \label{Lem stable subalgebra reflects fh}
    Let $\A$ be a modal algebra and $\B$ be a subalgebra of $\A$. If $\B$ is of finite height, then so is $\A$.
\end{lemma}

\begin{proof}
    Let $i: \B \to \A$ be an embedding. Suppose $\B$ is of finite height. Then there is some $n \in \omega$ such that $\Box^n 0_\B = 1_\B$. So, $i(\Box^n 0_\B) = i(1_\B)$, and hence $\Box^n 0_\A = 1_\A$ since $i$ is an embedding. Thus, $\A$ is of finite height.
\end{proof}

\begin{remark}
    In the proof, it suffices to assume that $i(\Box b) \leq \Box i(b)$ for all $b \in \B$. So, we may weaken the assumption so that $\B$ is a \emph{stable subalgebra} of $\A$ (cf. \cite[Definition 2.1]{stablemodallogic}). But this is not needed for our purpose. 
\end{remark}

\begin{theorem} \label{Thm characterization of K-union-splittings}
    For any modal logic $L$, the following are equivalent:
    \begin{enumerate}
        \item $L$ is a union-splitting in $\NExt{\K}$,
        \item $L$ is axiomatized over $\K$ by Jankov formulas of finite s.i.~modal algebras of finite height,
        \item For any modal algebra $\A$, $\H(\A)\si \cap \FH \subseteq \V(L)$ implies $\A \in \V(L)$,
        \item  For any finite modal algebra $\A$, $\H(\A)\fsi \cap \FH \subseteq \V(L)$ implies $\A \in \V(L)$,
    \end{enumerate}
\end{theorem}

\begin{proof} \leavevmode \par
    $(1) \Leftrightarrow (2)$: This is the second item of \Cref{Thm union splitting and Jankov formula}.
    
    $(2) \Rightarrow (3)$: Suppose that $L = \K + \{\epsilon(\B_i): i \in I\}$ where each $\B_i$ is a finite s.i.~modal algebra of finite height. Let $\A$ be a modal algebra such that $\A \not\models L$. We show that $\H(\A)\si \cap \FH \not\subseteq \V(L)$. Since $\A \not\models L$, $\A \not\models \epsilon(\B_i)$ for some $i \in I$. So, $\B_i$ is a subalgebra of a s.i.~homomorphic image $\A'$ of $\A$. Since $\B_i$ is of finite height, it follows from \Cref{Lem stable subalgebra reflects fh} that $\A'$ is also of finite height. Thus, $\A' \in \H(\A)\si \cap \FH$. However, since $\B_i$ is a subalgebra of $\A'$ and $\A'$ is a s.i.~homomorphic image of itself, $\A' \not\models \epsilon(\B_i)$. Hence, $\A' \not\models L$, namely, $\A' \notin \V(L)$. So, $\H(\A)\si \cap \FH \not\subseteq \V(L)$. 
    
    $(3) \Rightarrow (4)$: This is clear because for any finite modal algebra $\A$, $\H(\A)\fsi = \H(\A)\si$.

    $(4) \Rightarrow (2)$: Suppose that (4) holds. Let 
    \[L' = \K + \{\epsilon(\B): \B \in \FH\fsi, \B \not\models L\}.\] 
    It suffices to show that $L = L'$.

    Since $L'$ is a union-splitting by \Cref{Thm union splitting and Jankov formula}, it has the finite model property by \Cref{Thn K-union-splittings have fmp}. If $L \not\subseteq L'$, then by the finite model property of $L'$, there is a finite modal algebra $\A$ such that $\A \models L'$ and $\A \not\models L$. So, by (4), there is some $\B \in \H(\A)\fsi \cap \FH$ such that $\B \not\models L$. Since $\B \not\models \epsilon(\B)$ and $\epsilon(\B) \in L'$ by the definition of $L'$, $\B \not\models L'$. Since $\H$ preserves validity, we have $\A \not\models L'$, which is a contradiction. Thus, $L \subseteq L'$. 
    
    If $L' \not\subseteq L$, then there is a modal algebra $\A$ such that $\A \models L$ and $\A \not\models L'$. By the definition of $L'$, $\A \not\models \epsilon(\B)$ for some $\B \in \FH\fsi$ such that $\B \not\models L$. So, $\B$ is a subalgebra of a s.i.~homomorphic image $\A'$ of $\A$. Since $\H$ and $\S$ preserve validity, $\A \not\models L$, which is a contradiction. Thus, $L' \subseteq L$, and therefore, $L = L'$. 
\end{proof}

In practice, the characterization is particularly useful when combined with modal duality. Recall that a finite modal algebra is of finite height iff its dual space is cycle-free. 

\begin{example} \leavevmode
    \begin{enumerate}
        \item $\KD = \K + \Dia \top$ is the largest union-splitting in $\NExt{\K}$. 

        Let $\A$ be a finite modal algebra such that $\A \not\models \KD$ and $\X = (X, R)$ be its dual space. Then, there is a point $x \in \X$ such that $\X, x \not\models \Dia \top$, that is, $x$ is a dead end in $\X$. So, $\{x\}$ is a finite rooted cycle-free closed upset of $\X$, and thus corresponds to an algebra $\A' \in \H(\A)\fsi \cap \FH$. Also, it is clear that $\A' \not\models \KD$ since $x$ is a dead end. Thus, for any finite modal algebra $\A$, $\A \notin \V(\KD)$ implies $\H(\A)\fsi \cap \FH \not\subseteq \V(\KD)$. Hence, $\KD$ is a union-splitting by \Cref{Thm characterization of K-union-splittings}.
        
        Moreover, let $L$ be a union-splitting in $\NExt{\K}$. For any finite $\KD$-algebra $\A$, since the dual space of $\A$ is serial, there is no cycle-free closed upset of $\A$, which implies that $\H(\A)\fsi \cap \FH = \emp$, thus $\A \in \V(L)$ by \Cref{Thm characterization of K-union-splittings}. Since $\KD$ has the finite model property, it follows that $\V(\logic{KD}) \subseteq \V(L)$, namely, $L \subseteq \logic{KD}$. So, $\logic{KD}$ is the largest union-splitting in $\NExt{\K}$.
        
        \item $\K + \Box \Dia \top$ is not a union-splitting in $\NExt{\K}$. Let $\X$ be the finite modal space $\kdsecond$ and $\A$ be its dual algebra. Clearly, $\A \not\models \Box\Dia\top$. On the other hand, since the only cycle-free upset of $\X$ is the irreflexive singleton, which validates $\Box\Dia\top$, it holds that $\H(\A)\fsi \cap \FH \subseteq \V(\K + \Box \Dia \top)$. So, $\K + \Box \Dia \top$ is not a union-splitting by \Cref{Thm characterization of K-union-splittings}.
    \end{enumerate}
\end{example}

The condition (4) in \Cref{Thm characterization of K-union-splittings} is special in its finitary nature. Subframe logics (\cite{fineK4II}, see also \cite[Section 11.3]{czModalLogic1997}) and stable logics (\cite{stablemodallogic}) are two kinds of logics that have similar semantic characterizations. However, the following example shows that there is a logic $L \in \NExt{\Kf}$ such that the class of finite rooted $L$-spaces is closed under subframes, while $L$ is not a subframe logic. Also, several characterizations of stable logics are obtained in \cite{stablemodallogic} (see also \cite{FiltrationRevisitedLattices2018}), but none of them is completely finitary like the condition (4) in \Cref{Thm characterization of K-union-splittings}. An example for stable logics similar to the one below can be found in \cite[Example 5.9]{TakahashiThesis}.

\begin{example}
    Let $\F$ be the Kripke frame of negative integers with the order $<$ and a reflexive root $\omega$ at the bottom. Let $L = \Log\F$.
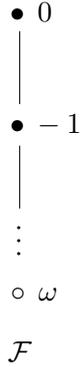
    \begin{figure}[H]
        \centering
        \begin{tikzpicture}[scale=1.5]
            \node (0) at (0,0) {\(\quad \bullet \:\: 0 \:\)};
            \node (-1) at (0,-1) {\(\quad \quad \bullet \: -1 \:\)};
            \node (-2) at (0,-2) {\( \rotatebox[origin=c]{90}{$\cdots$}\)};
            \node (-3) at (0,-2.5) {\(\quad \circ \:\: \omega \)};
            \draw(-1) -- (0);
            \draw(-2) -- (-1);
            \node at (0,-3) {\(\F\)};
        \end{tikzpicture}
    \end{figure}
    
    For any finite rooted frame $\G$, if $\G \models L$, then $\gamma(\G) \notin L$ since $\G \not\models \gamma(\G)$. So, $\F \not\models \gamma(\G)$, which means that $\G$ is a p-morphic image of a generated subframe of $\F$. Note that a p-morphism cannot identify two irreflexive points in $\F$. So, since $\G$ is finite, $\G$ must be a finite irreflexive chain. Thus, the class of finite rooted $L$-frames is the class of finite irreflexive chains, which is closed under subframes. However, $\F$ has the single reflexive point as its subframe, which refutes $L$. So, the class of $L$-frames is not closed under subframes, hence $L$ is not a subframe logic. 
\end{example}

Now we turn to our main theorem, the decidability of being a union-splitting in $\NExt{\K}$. This answers the open question \cite[Problem 2]{handbookModalDecision2007} in the affirmative. 

\begin{lemma} \label{Lem finite union splitting}
    Let $L_0$ be a logic and $\phi$ be a formula. If $L_0 + \phi = L_0 + \{\psi_i: i \in I\}$, then there is a finite subset $I' \subseteq I$ such that $L_0 + \phi = L_0 + \{\psi_i: i \in I'\}$.
\end{lemma}

\begin{proof}
    Let $\Phi$ be a set of formulas. It is clear from the syntax that, for any $\psi \in L_0 + \Phi$, there is a finite subset $\Phi' \subseteq \Phi$ such that $\psi \in L_0 + \Phi'$. It follows that the closure operator $\Phi \mapsto L_0 + \Phi$ on the set of formulas is algebraic and $\NExt{L_0}$ is an algebraic lattice, and finitely axiomatizable logics are exactly compact elements in $\NExt{L_0}$ (see, e.g., \cite[Theorem 2.30]{UniversalAlgebraFundamentals2011}). 

    If $L_0 + \phi = L_0 + \{\psi_i: i \in I\}$, then since it is a compact element, there is a finite subset $I' \subseteq I$ such that $L_0 + \phi \subseteq L_0 + \{\psi_i: i \in I'\}$. Since $L_0 + \{\psi_i: i \in I'\} \subseteq L_0 + \{\psi_i: i \in I\} = L_0 + \phi$, we obtain $L_0 + \phi = L_0 + \{\psi_i: i \in I'\}$.
\end{proof}

\begin{lemma} \label{Lem decidable fh}
    It is decidable whether a finite modal algebra is of finite height.
\end{lemma}

\begin{proof}
    Let $\A$ be a finite modal algebra. For any $a \in \A$ and any $n, m \in \omega$ such that $n < m$, if $\Box^n a = \Box^m a$, then $\{\Box^{m+k}a: k \in \omega\} = \{\Box^k a: n \leq k \leq m-1\}$. So, $\{\Box^n a : n \in \omega\} = \{\Box^n a : n < |A|\}$ for any $a \in \A$. It follows that $\A$ is of finite height iff there is an $n < |A|$ such that $\Box^n 0 = 1$, which is decidable. 
\end{proof}

\begin{lemma} \label{Lem decidable si}
    It is decidable whether a finite modal algebra is s.i.
\end{lemma}

\begin{proof}
    Let $\A$ be a finite modal algebra. We use the characterization in \cite{rautenbergSplittingLatticesLogics1980}: $\A$ is s.i. iff $\A$ has an \emph{opremum}, that is, an element $c \neq 1$ such that for any $a \neq 1$, there is some $n \in \omega$ such that $\Boxx{\leq n} a \leq c$. By the same argument as in the proof of \Cref{Lem decidable fh}, if $\Box^{\leq n} a \leq c$ for some $n \in \omega$, then there must be such an $n < |A|$. So, the existence of an opremum is decidable, and hence so is being s.i. 
\end{proof}

\begin{theorem} \label{Thm decidable union splitting}
    Being a union-splitting is decidable in $\NExt{\K}$. That is, it is decidable, given a formula $\phi$, whether the logic $\K + \phi$ is a union-splitting in $\NExt{\K}$.
\end{theorem}

\begin{proof}
    First, we show that the union-splitting problem is $\Sigma^0_1$. By \Cref{Thm characterization of K-union-splittings} and \Cref{Lem finite union splitting}, $\K + \phi$ is a union-splitting iff there is a finite set $\{\A_i: i < n\}$ of finite s.i.~modal algebras of finite height such that $\K + \phi = \K + \{\epsilon(\A_i): i < n\}$. By \Cref{Lem decidable fh} and \Cref{Lem decidable si}, it is decidable whether a finite modal algebra is s.i. and of finite height. So, finite s.i.~modal algebras of finite height and finite sets of them can be effectively enumerated. Moreover, since a finitely axiomatized logic is recursively enumerable, given $\phi$ and a finite set $\{\A_i: i < n\}$ of finite s.i.~modal algebras of finite height, the problem whether $\K + \phi = \K + \{\epsilon(\A_i): i < n\}$ is $\Sigma^0_1$. As $\Sigma^0_1$ is closed under existential quantification, it follows that the union-splitting problem is $\Sigma^0_1$.

    Next, we show that the union-splitting problem is $\Pi^0_1$. By \Cref{Thm characterization of K-union-splittings}, $\K + \phi$ is not a union-splitting iff there is a finite modal algebra $\A$ such that $\H(\A)\fsi \cap \FH \subseteq \V(\K + \phi)$ and $\A \notin \V(K + \phi)$. A finite modal algebra $\A$ only has finitely many s.i.~homomorphic images of finite height, which can be computed by \Cref{Lem decidable fh} and \Cref{Lem decidable si}. So, whether $\K + \phi$ is not a union-splitting is $\Sigma^0_1$, hence the union-splitting problem is $\Pi^0_1$. 

    Thus, the union-splitting problem is both $\Sigma^0_1$ and $\Pi^0_1$, hence decidable.
\end{proof}

We can also provide an intuitive description of an algorithm that decides union-splittings in $\NExt{\K}$ as follows. We start enumerating all finite sets $\{\epsilon(\A_i): i < n\}$ where each $\A_i$ is a finite s.i.~modal algebra of finite height. During the enumeration, for each enumerated finite set, we start verifying whether $\K + \phi = \K + \{\epsilon(\A_i): i < n\}$ holds. If $\K + \phi$ is a union-splitting, then eventually the enumeration will find a finite set $\{\epsilon(\A_i): i < n\}$ that axiomatizes $\K + \phi$ and the identification verification halts. (One might be concerned about the ``nested'' computation here, but this is fine because we have a computable bijection from $\omega$ to $\omega \times \omega$, which is the main reason that $\Sigma^0_1$ is closed under existential quantification.) Simultaneously, we start enumerating all finite modal algebras. For each of them, we compute all its s.i.~homomorphic images of finite height and check if any of them witnesses that $\K + \phi$ breaks the condition (4) in \Cref{Thm characterization of K-union-splittings}. If $\K + \phi$ is not a union-splitting, we will eventually find such a witness. Combining these two, the algorithm decides whether $\K+\phi$ is a union-splitting. 

This algorithm has an advantage that it is constructive, in the sense that if $\K + \phi$ is a union-splitting, then the algorithm outputs a finite set $\{\epsilon(\A_i): i < n\}$ that axiomatizes $\K + \phi$. From this, the decidability of being a splitting in $\NExt{\K}$ follows. 

\begin{theorem} \label{Thm decidable K-splitting}
    Being a splitting is decidable in $\NExt{\K}$. That is, it is decidable, given a formula $\phi$, whether the logic $\K + \phi$ is a splitting in $\NExt{\K}$.
\end{theorem}

\begin{proof}
    First, we run the algorithm from \Cref{Thm decidable union splitting}. If $\K + \phi$ is not a union-splitting, then the algorithm outputs false, and we are done because $\K + \phi$ is not a splitting. Suppose that $\K + \phi$ is a union-splitting. Then the algorithm outputs an axiomatization $\K + \phi = \K + \{\epsilon(\A_i): i < n\}$, where each $\A_i$ is a finite s.i.~modal algebra of finite height. If $n=0$, then $\K + \phi = \K$, which is not a splitting. So, we may assume $n \geq 1$.

    If $\K + \phi$ is a splitting, then $\K + \phi = \K + \epsilon(\B)$ for some finite s.i.~modal algebra $\B$ of finite height. For each $\A_i$, since $\A_i \not\models \epsilon(\A_i)$, we have $\A_i \not\models \epsilon(\B)$, so there is a homomorphic image $\A_i'$ of $\A_i$ such that $\B$ is a subalgebra of $\A_i'$, thus $|B| \leq |A_i|$. Let $m = \max\{|\A_i|: i < n\}$. Then $|B| \leq m$. Thus, $\K + \phi$ is a splitting iff $\K + \{\epsilon(\A_i): i < n\} = \K + \epsilon(\B)$ for some finite s.i.~modal algebra $\B$ of finite height such that $|B| \leq m$. 

    There are only finitely many such $\B$, and we can effectively enumerate all of them. Also, given a $\B$, whether $\K + \{\epsilon(\A_i): i < n\} = \K + \epsilon(\B)$ holds is decidable because both logics are decidable by \Cref{Cor fa K-union-splittings are decidable}. Thus, it is decidable whether $\K + \phi$ is a splitting. 
\end{proof}

Note that in other lattices of normal modal logics, say, $\NExt{\Kf}$ or $\NExt{\Sf}$, the situation is very different. All finite $\Kf$-algebras (resp. $\Sf$-algebras), not only those of finite height, split the lattice $\NExt{\Kf}$ (resp. $\NExt{\Sf}$). It is unknown if \Cref{Thn K-union-splittings have fmp} holds in $\NExt{\Kf}$ or $\NExt{\Sf}$, let alone the decidability of union-splittings.

\section{Axiomatization problems and (un)decidable formulas}

In this section, we apply our decidability result to axiomatization problems and (un)decidable formulas. 

\subsubsection*{Axiomatization problems}

Axiomatization problems are one of the simplest types of decision problems and have been widely studied (see, e.g., \cite[Chapter 17]{czModalLogic1997} and \cite{handbookModalDecision2007}). They are called the \emph{problem of coincidence} in \cite[Chapter 17]{czModalLogic1997}

\begin{definition}
    Given a modal logic $L_0$ and a formula $\phi$, the \emph{axiomatization problem} for $L_0 + \phi$ is, given a formula $\psi$, to decide whether $L_0 + \psi = L_0 + \phi$.
\end{definition}
In other words, the axiomatization problem for $L_0 + \phi$ is decidable iff the property ``$ = L_0 + \phi$'' is decidable in $\NExt{L_0}$. 

The following theorem can be proved by combining Chagrov's method \cite{chagrov1990I, chagrov1990II} and the proof of Blok's dichotomy theorem \cite{blok1978degree}. This is claimed in \cite[Section 17.6]{czModalLogic1997} and a proof can be found in \cite[Theorem 7]{handbookModalDecision2007}.

\begin{theorem} \label{Thm ax problem iff K-union-splitting}
    The axiomatization problem for a logic $\K + \phi$ is decidable iff $\K + \phi$ is a union-splitting in $\NExt{\K}$ or the inconsistent logic.
\end{theorem}

We provide a proof of the right-to-left direction using our algorithm that decides union-splittings. Because of the constructive nature of the algorithm, unlike the proof in \cite[Theorem 7]{handbookModalDecision2007}, our proof yields an algorithm that, given a formula $\phi$ such that $\K + \phi$ is a union-splitting in $\NExt{\K}$ or the inconsistent logic, outputs an algorithm that decides the axiomatization problem for $\K + \phi$. 

\begin{lemma} \label{Lem ax problem for K-union-splitting}
    Let $\K + \phi$ be a union-splitting in $\NExt{\K}$. Then the axiomatization problem for $\K + \phi$ is decidable. 
\end{lemma}

\begin{proof}
    Let $\K + \phi$ be a union-splitting in $\NExt{\K}$. Applying the algorithm from \Cref{Thm decidable union splitting}, we obtain an axiomatization $\K + \phi = \K + \{\epsilon(\A_i): i < n\}$, where each $\A_i$ is a finite s.i.~modal algebra of finite height. If $n = 0$, then $\K + \phi = \K$, and the decidability of the axiomatization problem for $\K + \phi$ follows from the decidability of $\K$. So, we may assume $n \geq 1$. 
    
    Given a formula $\psi$, $\K + \psi = \K + \phi$ iff $\psi \in \K + \{\epsilon(\A_i): i < n\}$ and $\epsilon(\A_i) \in \K + \psi$ for $i < n$. Whether $\psi \in \K + \{\epsilon(\A_i): i < n\}$ holds is decidable by \Cref{Cor fa K-union-splittings are decidable}. To decide whether $\epsilon(\A_i) \in \K + \psi$ holds, note that if $\epsilon(\A_i) \in \K + \psi$ then $\A_i \not\models \psi$ since $\A_i \not\models \epsilon(\A_i)$, and if $\A_i \not\models \psi$, then $\K + \psi \not\subseteq \Log\A_i$, so $\K + \epsilon(\A_i) \subseteq \K + \psi$ since $(\K + \epsilon(\A_i), \Log\A_i)$ is a splitting pair, and thus $\epsilon(\A_i) \in \K + \psi$. So, it suffices to check whether $\A_i \not\models \psi$ holds, which is decidable. Hence, it is decidable whether $\K + \psi = \K + \phi$.
\end{proof}

\begin{lemma} \label{Lem ax problem for inconsistent}
    The axiomatization problem for the inconsistent logic is decidable.
\end{lemma}

\begin{proof}
    It is well-known that the consistency is decidable in $\NExt{\K}$ (see, e.g., \cite[Theorem 17.2]{czModalLogic1997}, as the lattice $\NExt{\K}$ has only two co-atoms and both are decidable. The statement follows as the axiomatization problem for the inconsistent logic coincides with the problem of deciding inconsistency. 
\end{proof}

\begin{proof}[Proof of \Cref{Thm ax problem iff K-union-splitting}]
    The right-to-left direction follows from \Cref{Lem ax problem for K-union-splitting} and \Cref{Lem ax problem for inconsistent}; see \cite[Theorem 7]{handbookModalDecision2007} for the other direction. 
\end{proof}

\begin{corollary}
    It is decidable whether the axiomatization problem for $\K + \phi$ is decidable.
\end{corollary}

\begin{proof}
    This follows from \Cref{Thm decidable union splitting}, \Cref{Thm ax problem iff K-union-splitting}, and the fact that consistency is decidable.
\end{proof}

Moreover, the proof yields an algorithm that not only decides whether an axiomatization problem is decidable for a logic $\K + \phi$, but also outputs an algorithm deciding the axiomatization problem. 

\subsubsection*{(Un)decidable formulas}

Undecidable formulas are introduced in \cite{Chagrov1994} (see also \cite[Section 16.4]{czModalLogic1997}).

\begin{definition}
    Let $L_0$ be a logic. A formula $\phi$ is called a \emph{(un)decidable formula} in $\NExt{L_0}$ if it is (un)decidable, given a formula $\psi$, whether $\phi \in L + \psi$.
\end{definition}

From the proof of \Cref{Lem ax problem for K-union-splitting}, we observe that if $\K + \phi$ is a union-splitting in $\NExt{\K}$, then $\phi$ is a decidable formula in $\NExt{\K}$: we first compute an axiomatization $\K + \phi = \K + \{\epsilon(\A_i): i < n\}$, then for any formula $\psi$, $\phi \in \K + \psi$ iff $\A_i \not\models \psi$ for all $i$, which is a decidable condition. Also, $\bot$ is a decidable formula in $\NExt{\K}$ as the consistency is decidable. Moreover, the proof of \Cref{Thm ax problem iff K-union-splitting} in \cite[Theorem 7]{handbookModalDecision2007} in fact established the converse: if $\K + \phi$ is neither a union-splitting nor the inconsistent logic in $\NExt{\K}$, then $\phi$ is an undecidable formula in $\NExt{\K}$. Thus, we obtain the following theorem.

\begin{theorem} \label{Thm decidable fml iff K-union-splitting}
    A formula $\phi$ is a decidable formula in $\NExt{\K}$ iff $\K + \phi$ is a union-splitting in $\NExt{\K}$ or the inconsistent logic.
\end{theorem}

\begin{corollary}
    It is decidable whether $\phi$ is a (un)decidable formula in $\NExt{\K}$.
\end{corollary}

\begin{proof}
    This follows from \Cref{Thm decidable union splitting}, \Cref{Thm decidable fml iff K-union-splitting}, and the fact that consistency is decidable.
\end{proof}

This answers \cite[Problem 17.3]{czModalLogic1997} for $\NExt{\K}$ in the affirmative.

Moreover, we also obtain a somewhat mysterious equivalence between decidable axiomatization problems, decidable formulas, and union-splittings in $\NExt{\K}$. Given that the proof of the undecidability of various logical properties \cite[Theorem 9]{handbookModalDecision2007} makes an essential use of union-splittings, union-splittings may have a deeper connection to the decision problem of logical properties.

\begin{corollary} \label{Cor equivalent}
    For any formula $\phi$, the following are equivalent:
    \begin{enumerate}
        \item the axiomatization problem for $\K + \phi$ is decidable,
        \item $\phi$ is a decidable formula,
        \item $\K + \phi$ is a union-splittings in $\NExt{\K}$ or the inconsistent logic.
    \end{enumerate}
\end{corollary}

\begin{proof}
    This follows from \Cref{Thm ax problem iff K-union-splitting} and \Cref{Thm decidable fml iff K-union-splitting}.
\end{proof}

\begin{remark}
    A Kripke complete logic $L$ is called \emph{strictly Kripke complete} (see, e.g., \cite[Section 10.5]{czModalLogic1997}) if no other logic has the same class of Kripke frames as $L$. Blok \cite{blok1978degree} showed that a logic $L$ is strictly Kripke complete iff $L$ is a union-splitting or the inconsistent logic. Thus, strictly Kripke completeness is equivalent to each item in \Cref{Cor equivalent} and is also decidable. 
\end{remark}

\subsection*{}
\begin{acknowledgements}
The author is very grateful to Nick Bezhanishvili for his supervision on the Master's thesis and his valuable comments on this paper. The author was supported by the Student Exchange Support Program (Graduate Scholarship for Degree Seeking Students) of the Japan Student Services Organization and the Student Award Scholarship
of the Foundation for Dietary Scientific Research.
    
\end{acknowledgements}

\printbibliography[heading=bibintoc,title=References]

\end{document}

%% file: headers.tex
\usepackage{authblk}

\usepackage{amsmath}
\usepackage{amsthm}
\usepackage{amssymb}
\usepackage{amsfonts}
\usepackage{mathtools}
\usepackage{bm}
\usepackage{bussproofs}

\usepackage{cancel}
\usepackage{lscape}

\usepackage{float}
\usepackage{booktabs}
\usepackage{caption}
\usepackage{adjustbox}

\usepackage{graphicx}
\usepackage{tikz}
\usetikzlibrary{cd,backgrounds,positioning,trees,shapes,arrows,patterns,topaths,calc}

\usepackage[
  backend=biber,
  style=alphabetic,
  natbib=true,
  maxcitenames=9,
  maxbibnames=99,
  abbreviate=false,
  eprint=false,
  doi=true,
]{biblatex}
\DeclareFieldFormat{doi}{\url{https://doi.org/#1}}
\DeclareFieldFormat{url}{\url{#1}}
\usepackage{hyperref}
\hypersetup{
  hidelinks,
  linktoc = all,
  breaklinks = true,
}
\usepackage[nameinlink]{cleveref}

\newtheorem{theorem}{Theorem}
\newtheorem{lemma}[theorem]{Lemma}
\newtheorem{corollary}[theorem]{Corollary}

\newtheorem{proposition}[theorem]{Proposition}

\theoremstyle{definition}
\newtheorem{definition}[theorem]{Definition}
\newtheorem{remark}[theorem]{Remark}
\newtheorem{example}[theorem]{Example}

\theoremstyle{remark}

\numberwithin{theorem}{section}

\newcommand{\alg}[1]{\mathfrak{#1}}
    \newcommand{\A}{\alg{A}}
    \newcommand{\B}{\alg{B}}

\renewcommand{\sp}[1]{\mathfrak{#1}}
    \newcommand{\X}{\sp{X}}

\newcommand{\fr}[1]{\mathcal{#1}}
    \newcommand{\F}{\fr{F}}
    \newcommand{\G}{\fr{G}}

\newcommand{\class}[1]{\mathcal{#1}}
    \newcommand{\V}{\class{V}}

    \renewcommand{\S}{\class{S}}
    
    \newcommand{\FH}{\sf{FH}}

\newcommand{\si}{_\mathrm{si}}
\newcommand{\fsi}{_\mathrm{fsi}}

\newcommand{\Log}{\mathsf{Log}}

\newcommand{\logic}[1]{\mathsf{#1}}

    \newcommand{\K}{\logic{K}}
    \newcommand{\Kf}{\logic{K4}}
    
    \newcommand{\Sf}{\logic{S4}}

    \newcommand{\KD}{\logic{KD}}

\newcommand{\NExt}[1]{\mathsf{NExt}{#1}}

\renewcommand{\H}{\mathsf{H}}

\renewcommand{\phi}{\varphi}
\newcommand{\emp}{\emptyset}

\newcommand{\goodbox}{\hspace{.2ex}\text{%
  \tikz[baseline=-.6ex, rounded corners=.01ex, line width=.12ex]
    {\draw(-.6ex,-.6ex) rectangle (.6ex,.6ex);}}\kern.2ex}
\newcommand{\gooddiamond}{\hspace{.2ex}\text{%
  \tikz[baseline=-.6ex, rounded corners=.01ex, rotate=45, line width=.12ex]
    {\draw(-.5ex,-.5ex) rectangle (.5ex,.5ex);}}\kern.2ex}
\renewcommand{\Box}{\goodbox}
\renewcommand{\Diamond}{\gooddiamond}

\newcommand{\Dia}{\Diamond}

\newcommand{\Boxx}[1]{\Box^{#1}}

\newcommand{\Lor}{\bigvee}
\newcommand{\Land}{\bigwedge}

\def\kdsecond{\tikz[baseline=.1ex]{
\draw(0,0.6ex) circle (2pt) coordinate (A);
\fill (3ex,0.6ex) circle (2pt) coordinate (B);
\draw[->, shorten <=0.4ex, shorten >=0.4ex] (A)-> (B);}
}

\newenvironment{acknowledgements}{%
  \begin{abstract}
}{%
  \end{abstract}
}

\renewcommand{\S}{\mathcal{S}}
\renewcommand{\H}{\mathcal{H}}
\renewcommand{\FH}{\mathcal{FH}}